\documentclass[12pt]{article}

\usepackage{amsmath,amssymb,amsthm,amssymb}
\usepackage{tikz,tkz-graph,tikz-cd,tikz-qtree}
\usepackage{youngtab}
\usepackage[OT2,T1]{fontenc}
\usepackage{indentfirst}
\usepackage[hidelinks]{hyperref}
\usepackage{booktabs}
\usepackage{float}
\usepackage{mathrsfs}

\newtheorem{theorem}{Theorem}[section]
\newtheorem{lemma}[theorem]{Lemma}
\newtheorem{corollary}[theorem]{Corollary}

\newtheorem{definition}[theorem]{Definition}
\newtheorem{proposition}[theorem]{Proposition}
\newtheorem{remark}[theorem]{Remark}

\newtheorem{examples}[theorem]{Example}

\begin{document}
\title{The Optimal Coefficients-Based Criterion for Primitive Quadratic Polynomials over Finite Fields}
\author{Hongfeng Wu$^{1}$ and Li Zhu$^{2}$\footnote{Corresponding author.}
\setcounter{footnote}{-1}
\footnote{E-mail addresses: whfmath@gmail.com (H. Wu), lizhumath@pku.edu.cn (L. Zhu).
}
\\[0.5ex]
\small $^{1}$College of Science, North China University of Technology, Beijing, China
\\
\small $^{2}$School of Mathematical Sciences, Guizhou Normal University, Guiyang, China}
	
\date{}
\maketitle

\thispagestyle{plain}
\setcounter{page}{1}
	
\begin{abstract}
	Let $\mathbb F_q$ be a finite field and consider quadratic polynomials
	$f(X)=X^2+bX+c\in\mathbb F_q[X]$ with primitive constant term $c$.
	We construct an optimal coefficients-based determining polynomial for
	primitive quadratic polynomials over every finite field. More precisely,
	for every primitive $c\in\mathbb F_q$, its specialization is the unique
	monic square-free polynomial whose roots are exactly the coefficients
	$b$ for which $X^2+bX+c$ is primitive. The construction is based on
	Lucas polynomials and Lucas atoms over finite fields. In odd
	characteristic, the optimal determining polynomial is the $(q+1)$-st
	Lucas atom. In characteristic $2$, this Lucas atom has multiplicity
	$2$ in the variable $B$, and the optimal polynomial is obtained by
	removing this multiplicity through the inverse Frobenius over
	$A=\mathbb F_q[C]/(\Phi_{q-1}(C))$. We prove that the determining
	polynomial is unique for each fixed primitive constant term and,
	globally, unique as an element of $A[B]$. We also give equivalent
	criteria involving only recursively computable Lucas polynomials and,
	as an application, a first-zero coefficient description of the
	binomial order and order of an irreducible quadratic polynomial.
\end{abstract}



\section{Introduction}

Primitive polynomials form an important class of polynomials over
finite fields. A monic irreducible polynomial of degree $n$ over
$\mathbb F_q$ is called primitive if any of its roots generates the
multiplicative group $\mathbb F_{q^n}^{*}$. Equivalently, an
irreducible polynomial $f(X)$ of degree $n$ is primitive if and only
if $\operatorname{ord}(f)=q^n-1$.

Primitive polynomials arise, for example, in the construction of
finite field extensions and maximal-period linear recurring
sequences, as well as in applications of finite fields to coding
theory and cryptography. We refer to \cite{LidlNiederreiter} for
general background.

Various criteria for primitive polynomials have been obtained.
Fitzgerald \cite{Fitzgerald} gave a characterization in terms of
polynomial division. More recently, Vega \cite{VegaCharacterization} studied the
quadratic case and obtained a characterization and an explicit
description of all primitive quadratic polynomials over finite fields.
One of the tools involved is the binomial order. For an irreducible quadratic polynomial
\[
f(X)=X^2+bX+c,
\]
with primitive constant term $c$, primitivity is equivalent to its
binomial order attaining the value $q+1$.

The criteria in \cite{VegaCharacterization} involve a polynomial division process.
In \cite{VegaNecessary}, Vega further considered whether primitivity of a
quadratic polynomial could be detected directly from its coefficients.
More precisely, motivated by polynomial criteria depending only on
coefficients, one may ask whether there exists a polynomial expression
in $b$ and $c$ which determines whether $X^2+bX+c$ is primitive. Criteria of this type were obtained in \cite{VegaNecessary} for
several special forms of the parameter $q+1$.

In this paper we consider this problem for arbitrary finite fields.
We impose in addition an optimality condition on the determining
polynomial. Fix a primitive element $c\in\mathbb F_q$ and put
\[
\mathcal P_c
=
\left\{
b\in\mathbb F_q:
X^2+bX+c
\text{ is primitive over }\mathbb F_q
\right\}.
\]
An optimal determining polynomial $P_q(B,C)$ is required to satisfy
\begin{equation}\label{eq 11}
	P_q(B,c)
	=
	\prod_{b\in\mathcal P_c}(B-b)
\end{equation}
for every primitive $c$. Thus $P_q(B,c)$ is monic and square-free,
and its degree is the number of primitive quadratic polynomials with
the prescribed constant term. In particular, it has the smallest
possible degree among polynomials having $\mathcal P_c$ as their set
of roots.

Condition \eqref{eq 11} leads to the following uniqueness property. For each fixed
primitive $c$, the polynomial $P_q(B,c)$ is uniquely determined.
Globally, a representative in $\mathbb F_q[B,C]$ need not be unique,
since only the specializations at primitive values of $C$ are relevant.
Instead, let $A = \mathbb{F}_{q}[C]/(\Phi_{q-1}(C))$. We prove that the image of an optimal determining polynomial in
\[
A[B]
=
\mathbb F_q[B,C]/(\Phi_{q-1}(C))
\]
is unique. A geometric explanation for such uniqueness is given as follows. Consider the affine scheme
\[
\mathcal X
=
\mathbb A^1_{\mathbb F_q}
\times
\operatorname{Spec}A
=
\operatorname{Spec}A[B].
\]
Its $\mathbb F_q$-rational points are exactly
\[
\mathcal X(\mathbb F_q)
=
\left\{
(b,c):
b\in\mathbb F_q,\;
c\text{ is primitive in }\mathbb F_q
\right\}.
\]
And $A[B]$ is the affine coordinate ring of the scheme $\mathcal{X}$, whose elements can be viewed as regular functions on $\mathcal{X}$. Thus the above uniqueness can be formulated as that the optimal determining polynomial determines a unique
regular function on $\mathcal X$.

Our construction uses Lucas polynomials and Lucas atoms. Let
$\mathcal{U}_n(B,C)$ and $\mathcal{V}_n(B,C)$ be the two Lucas polynomials associated
with
\[
T^2+BT+C,
\]
and let $\Lambda_n(B,C)$ denote the $n$-th Lucas atom. Their atomic
decompositions are analogous to the cyclotomic decomposition of
$X^n-1$. We develop the required theory over finite fields, including
the Frobenius multiplicities which occur when the index is divisible
by the characteristic.

The connection with quadratic polynomials is obtained through
binomial order. If $f(X)=X^2+bX+c$ is irreducible with roots
$\theta$ and $\theta^q$, then
\[
U_n(b,c)
=
\frac{\theta^n-\theta^{nq}}
{\theta-\theta^q},
\]
and the binomial order of $f$ is the smallest positive integer $n$
for which $U_n(b,c)=0$. Moreover, for $n$ coprime to the
characteristic,
\[
\Lambda_n(b,c)=0
\quad\Longleftrightarrow\quad
\operatorname{ord}
\left(
\frac{\theta}{\theta^q}
\right)=n.
\]
Thus the Lucas atom $\Lambda_{q+1}$ detects the binomial order
$q+1$, which is the order relevant to primitive quadratic polynomials
with primitive constant term.

When $q$ is odd, we obtain
\[
P_q(B,C)=\Lambda_{q+1}(B,C).
\]
For every primitive $c\in\mathbb F_q$, its roots are exactly the
coefficients $b$ for which $X^2+bX+c$ is primitive, and the
specialization is monic and square-free.

When $q$ is even, the Lucas atom satisfies
\[
\Lambda_{q+1}(B,C)\in\mathbb F_q[B^2,C],
\]
so its roots in the $B$-variable occur with multiplicity $2$. To
obtain the square-free determining polynomial, we work over the
coefficient ring
\[
A=\mathbb F_q[C]/(\Phi_{q-1}(C)).
\]
The Frobenius map on $A$ is an automorphism, and its inverse gives a
canonical map
\[
\operatorname{Fr}^{-1}:A[B^2]\longrightarrow A[B].
\]
The optimal determining polynomial in characteristic $2$ is then
given by
\[
P_q(B,C)
=
\operatorname{Fr}^{-1}
\bigl(\Lambda_{q+1}(B,C)\bigr).
\]
Hence in both characteristics the determining polynomial is obtained
from the $(q+1)$-st Lucas atom, with the Frobenius multiplicity removed
in characteristic $2$. The special cases considered in
\cite{VegaNecessary} are included in this construction.

Since Lucas atoms are less convenient to compute directly than Lucas
polynomials, we also give equivalent criteria involving only Lucas
polynomials. These are obtained from least common multiples of suitable
Lucas polynomials and can be computed from their recurrence relations
without first computing $\Lambda_{q+1}$.

Finally, as an application of the Lucas polynomial framework, we
consider the polynomial-division criterion. For an irreducible
quadratic polynomial $f(X)=X^2+bX+c$, the coefficients in the
quotient
\[
\frac{X^{q+2}-cX}{f(X)}
\]
are shown to be the terms of the Lucas sequence associated to $f$.
Consequently, the position of the first zero coefficient determines
the binomial order and the order of $f$.

The paper is organized as follows. Section~2 recalls the basic
definitions concerning order, binomial order, and primitive
polynomials. Section~3 develops Lucas polynomials, Lucas atoms, and
their atomic decompositions over finite fields. Section~4 gives the
preliminary enumeration and binomial-order results. Section~5
constructs the optimal determining polynomial in odd and even
characteristic, gives the corresponding Lucas-polynomial criteria,
and proves its uniqueness. Section~6 presents the first-zero
coefficient approach as an application.
	
\section{Definitions, notations and basic results}
Throughout this note $q$ denotes a prime power. Let $\mathbb{F}_{q}$ be a finite field with $q$ elements, and $\mathbb{F}_{q}^{\ast}$ be the cyclic multiplicative group of nonzero elements in $\mathbb{F}_{q}$. Any generator of $\mathbb{F}_{q}^{\ast}$ is called a primitive element of $\mathbb{F}_{q}$. For $\lambda \in \mathbb{F}_{q}^{\ast}$, the order of $\lambda$ is defined to be the smallest positive integer $n$ such that $\lambda^{n}=1$, and is denoted by $\mathrm{ord}(\lambda)=n$.

Unless stating otherwise, in this paper polynomials are assumed to be nonconstant, monic and has a nonzero constant term. For such a polynomial $f(X) \in \mathbb{F}_{q}[X]$, there is a smallest positive integer $n$ such that $f(X) \mid X^{n}-1$. The integer $n$ is called the order of $f(X)$ and is denoted by $n=\mathrm{ord}(f)$. If $f(X)$ is irreducible over $\mathbb{F}_{q}$, then $\mathrm{ord}(f) = \mathrm{ord}(\gamma)$ for any root $\gamma$ of $f(X)$ lying in the finite extension field $\mathbb{F}_{q}[X]/(f(X))$ of $\mathbb{F}_{q}$.

Further, the minimal binomial multiple and the binomial order of $f(X)$ are defined as follow.

\begin{definition}\label{def 1}
	Define the minimal binomial multiple of $f(X)$ to be the monic binomial $X^{r}-\lambda \in \mathbb{F}_{q}[X]$ which is of the lowest degree among the binomials over $\mathbb{F}_{q}$ divisible by $f(X)$. The degree $r$ is called the binomial order of $f(X)$, and is denoted by $\mathrm{ord}_{\mathrm{b}}(f) = r$.
\end{definition}

In some references (for instance \cite{LidlNiederreiter}), quasi-order is used to refer to binomial order. We adopt the terminologies of Definition \ref{def 1}, since not only binomial order but also minimal binomial multiple is frequently involved. Some basic relations between the order and the minimal binomial multiple of $f(X)$ are summarized in the next lemma.

\begin{lemma}\label{lem 1}
	\begin{itemize}
		\item[(1)] Let $f(X)$ be a polynomial over $\mathbb{F}_{q}$, with order $n$, binomial order $r$ and minimal binomial multiple $X^{r}-\lambda$. Then $r \mid n$ and
		$$X^{r}-\lambda \mid X^{n}-1.$$
		\item[(2)] If $f(X)$ is irreducible and has order $n$, then its binomial order is $r = \frac{n}{\mathrm{gcd}(n,q-1)}$ and its minimal binomial multiple is given by
		$$X^{r}-\zeta^{r},$$
		where $\zeta$ is any root of $f(X)$.
		\item[(3)] If $f(X)$ is irreducible and has minimal binomial multiple $X^{r}-\lambda$, then its order is given by $n = r\cdot\mathrm{ord}(\lambda)$.
	\end{itemize}
\end{lemma}

\begin{definition}
	A polynomial $f(X) \in \mathbb{F}_{q}$ of degree $k \geq 1$ is said to be a primitive polynomial if it is the minimal polynomial of a primitive element of $\mathbb{F}_{q^{k}}$ over $\mathbb{F}_{q}$.
\end{definition}

The next two theorems are classical criteria on primitive polynomials, in terms of their order and binomial order respectively.

\begin{lemma}\label{lem 2}
	A polynomial $f(X) \in \mathbb{F}_{q}[X]$ of degree $k \geq 1$ is primitive if and only if $\mathrm{ord}(f) = q^{k}-1$.
\end{lemma}

\begin{lemma}\label{lem 4}
	Assume that $f(X) \in \mathbb{F}_{q}[X]$ is given by
	$$f(X) = X^{k}+a_{k-1}X^{k-1}+\cdots+a_{0},$$
	where $k \geq 1$. Then $f(X)$ is a primitive polynomial if and only if $(-1)^{k}a_{0}$ is a primitive element in $\mathbb{F}_{q}$ and $\mathrm{ord}^{\mathrm{b}}(f) = \frac{q^{k}-1}{q-1}$. If it is this case, the minimal binomial multiple of $f(X)$ is given by
	$$X^{\frac{q^{k}-1}{q-1}}-(-1)^{m}a_{0}.$$
\end{lemma}	
	
\section{Lucas polynomials, Lucas atomic decomposition, and Lucas sequences}\label{sec Luc}
Lucas polynomials were originally defined over $\mathbb{Q}$. Much in the same way that cyclotomic polynomials give the irreducible factorization of $X^{n}-1$ over $\mathbb{Q}$, the irreducible factorization of any Lucas polynomial is given by irreducible factors called Lucas atoms, and this factorization is accordingly referred to as the Lucas atomic decomposition of the Lucas polynomial. This result was first established in \cite{Sagan}. Since both Lucas polynomials and Lucas atoms have integer coefficients, their definitions and the Lucas atomic decompositions carry over to finite fields naturally.

Lucas polynomials over finite fields, along with their atomic decompositions, are the central tool for deriving the coefficients-based criterion for primitive quadratic polynomials, which is presented in Section \ref{sec: criterion}. In order to be self-contained and also to fix the notations, we will give the precise definitions of Lucas polynomials over finite fields, and prove their atomic decompositions in the first subsection. And in the second subsection we will introduce the Lucas sequences associated to an irreducible quadratic polynomial $f(X) \in \mathbb{F}_{q}[X]$, which relates the binomial order of $f(X)$ to Lucas polynomials.

\subsection{Lucas polynomials over finite fields and their atomic decomposition}
Let $\mathbb{F}_{q}$ be a finite field of $q$ elements, where $q$ is a power of a prime $p$, and let $B$ and $C$ be two algebraically independent indeterminates over $\mathbb{F}_{q}$. Denote by $\mathbf{k} = \mathbb{F}_{q}(B,C)$ the field of rational functions in $B$ and $C$. 

\begin{definition}
	We define two families $\{\mathcal{U}_{n}(B,C)\}_{n \geq 0}$ and $\{\mathcal{V}_{n}(B,C)\}_{n \geq 0}$ of polynomials over $\mathbb{F}_{q}$ inductively by:
	\begin{itemize}
		\item[(1)] $\mathcal{U}_{0}(B,C)=0$, $\mathcal{U}_{1}(B,C)=1$, and
		$$\mathcal{U}_{n}(B,C)= -B\,\mathcal{U}_{n-1}(B,C)-C\,\mathcal{U}_{n-2}(B,C)$$
		for $n \geq 2$;
		\item[(2)] $\mathcal{V}_{0}(B,C)=2$, $\mathcal{V}_{1}(B,C)=-B$, and
		$$\mathcal{V}_{n}(B,C)= -B\,\mathcal{V}_{n-1}(B,C)-C\,\mathcal{V}_{n-2}(B,C)$$
		for $n \geq 2$.
	\end{itemize}
	The polynomials $\mathcal{U}_{n}(B,C)$ and $\mathcal{V}_{n}(B,C)$ are called the $n$-th Lucas polynomials of the first type and of the second type, respectively.
\end{definition} 
	
Let $T$ be an indeterminate over $\mathbf{k}$, and $\alpha, \beta \in \overline{\mathbf{k}}$ be the roots of 
$$T^{2}+BT+C.$$
As $B$ and $C$ are algebraically independent, $B^{2}-4C \neq 0$ and hence the two roots $\alpha$ and $\beta$ are distinct. The polynomials $\mathcal{U}_{n}(B,C)$ and $\mathcal{V}_{n}(B,C)$ can be computed via the following Binet formulas.

\begin{lemma}\label{lem 3}
	The polynomials $\mathcal{U}_{n}(B,C)$ and $\mathcal{V}_{n}(B,C)$ can be written as
	$$\mathcal{U}_{n}(B,C) = \dfrac{\alpha^{n}-\beta^{n}}{\alpha-\beta}, \quad \mathcal{V}_{n}(B,C) = \alpha^{n}+\beta^{n}.$$
\end{lemma}
	
\begin{proof}
	Since $\alpha$ and $\beta$ are the roots of $T^{2}+BT+C$, then 
	$$\alpha+\beta = -B, \quad \alpha\beta = C.$$
	For $i = 0$ and $i=1$, by definition one has
	$$\dfrac{\alpha^{0}-\beta^{0}}{\alpha-\beta^{q}} = 0 = \mathcal{U}_{0}(B,C), \quad \dfrac{\alpha-\beta}{\alpha-\beta} = 1 = \mathcal{U}_{1}(B,C),$$
	and
	$$\alpha^{0}+\beta^{0} = 2 = \mathcal{V}_{0}(B,C), \quad \alpha+\beta = -b = \mathcal{V}_{1}(B,C).$$
	Suppose that the Binet formulas holds for all integers $\leq n$. Since 
	$$\alpha^{2} = -B\alpha-C, \quad \beta^{2} = -B\beta-C,$$
	then 
	$$\alpha^{n+1}=-B\alpha^n-C\alpha^{n-1}, \quad \beta^{n+1}=-B\beta^n-C\beta^{n-1},$$
	which yields
	\[
	\begin{aligned}
		\frac{\alpha^{n+1}-\beta^{n+1}}
		{\alpha-\beta}&=
		\frac{
			\left(-B\alpha^n-C\alpha^{n-1}\right)
			-
			\left(-B\beta^n-C\beta^{n-1}\right)}
		{\alpha-\beta}\\
		&=
		-B\cdot
		\frac{\alpha^n-\beta^{n}}
		{\alpha-\beta}
		-C\cdot
		\frac{\alpha^{n-1}-\beta^{n-1}}
		{\alpha-\beta}\\
		&=
		-B\,\mathcal{U}_{n}(B,C)-C\,\mathcal{U}_{n-1}(B,C).\\
		&=\mathcal{U}_{n+1}(B,C)
	\end{aligned}
	\]
	and
	\[
	\begin{aligned}
		\alpha^{n+1}+\beta^{n+1}&=
		-B\alpha^{n}-C\alpha^{n-1}-B\beta^{n}-C\beta^{n-1}\\
		&=
		-B(\alpha^{n}+\beta^{n})-C(\alpha^{n-1}+\beta^{n-1})\\
		&=
		-B\,\mathcal{V}_{n}(B,C)-C\,\mathcal{V}_{n-1}(B,C)\\
		&=\mathcal{V}_{n+1}(B,C).
	\end{aligned}
	\]
\end{proof}
	
\begin{corollary}\label{coro 1}
	For any $n \geq 1$, one has
	$$\mathcal{U}_{2n}(B,C) = \mathcal{U}_{n}(B,C)\mathcal{V}_{n}(B,C), \quad \mathcal{V}_{2n}(B,C) = \mathcal{V}_{n}(B,C)^{2} - 2C^{n}.$$
\end{corollary}
	
\begin{proof}
	Following from Lemma \ref{lem 3}, one has
	\[
	\begin{aligned}
		\mathcal{U}_{2n}(B,C) & = \dfrac{\alpha^{2n}-\beta^{2n}}{\alpha-\beta}\\
		& = \dfrac{(\alpha^{n}-\beta^{n})(\alpha^{n}+\beta^{n})}{\alpha-\beta}\\
		& = \mathcal{U}_{n}(B,C)\mathcal{V}_{n}(B,C),
	\end{aligned}
	\]
	and
	\[
	\begin{aligned}
		\mathcal{V}_{2n}(B,C) & = \alpha^{2n}+\beta^{2n}\\
		& = (\alpha^{n}+\beta^{n})^{2} - 2\alpha^{n}\beta^{n}\\
		& = \mathcal{V}_{n}(B,C)^{2} - 2C^{n}.
	\end{aligned}
	\]
\end{proof}
	
Let $\Phi_{n}(X)$ be the $n$-th cyclotomic polynomial over $\mathbb{F}_{q}$. The homogeneous $n$-th cyclotomic polynomial, in the indeterminates $\alpha$ and $\beta$, is defined by
$$\Phi_{n}^{\mathrm{hom}}(\alpha,\beta) = \beta^{\varphi(n)}\Phi_{n}(\frac{\alpha}{\beta}),$$
where $\varphi$ is the Euler's totient function. 

\begin{proposition}\label{prop 2}
	For any $n \geq 2$, there exists a unique polynomial $\Lambda_{n}(B,C) \in \mathbb{F}_{q}[B,C]$ such that
	$$\Lambda_{n}(B,C) = \Phi_{n}^{\mathrm{hom}}(\alpha,\beta).$$
\end{proposition}

\begin{proof}
	First assume that $n$ is not divisible by $p$. Then $\Phi_{n}(X)$ can be written as
	$$\Phi_{n}(X) = \prod_{\substack{1 \leq k \leq n\\ \mathrm{gcd}(k,n)=1}}(X-\zeta_{n}^{k}),$$
	where $\zeta_{n}$ is a primitive $n$-th root of unity lying in $\overline{\mathbb{F}}_{q}$, therefore 
	$$\Phi_{n}^{\mathrm{hom}}(\alpha,\beta) = \beta^{\varphi(n)}\Phi_{n}(\frac{\alpha}{\beta}) = \prod_{\substack{1 \leq k \leq n\\ \mathrm{gcd}(k,n)=1}}(\alpha-\zeta_{n}^{k}\beta).$$
	Since for $n\geq 2$, the cyclotomic polynomial $\Phi_n(X)$ is self-reciprocal, that is,
	\[
	X^{\varphi(n)}\Phi_n(X^{-1})=\Phi_n(X),
	\] 
	exchanging $\alpha$ and $\beta$ gives
	\[
	\begin{aligned}
		\Phi_n^{\mathrm{hom}}(\beta,\alpha)
		&=
		\alpha^{\varphi(n)}
		\Phi_n\left(\frac{\beta}{\alpha}\right)\\
		&=
		\beta^{\varphi(n)}
		\left(\frac{\alpha}{\beta}\right)^{\varphi(n)}
		\Phi_n\left(\frac{\beta}{\alpha}\right)\\
		&=
		\beta^{\varphi(n)}
		\Phi_n\left(\frac{\alpha}{\beta}\right)\\
		&=
		\Phi_n^{\mathrm{hom}}(\alpha,\beta).
	\end{aligned}
	\]
	Hence $\Phi_{n}^{\mathrm{hom}}(\alpha,\beta)$ is a symmetric polynomial in $\alpha$ and $\beta$. As $\alpha+\beta = -B$ and $\alpha\beta = C$, by the fundamental theorem of symmetric polynomials, there exists a unique polynomial $\Lambda_{n}(B,C) \in \mathbb{F}_{q}[B,C]$ such that
	$$\Lambda_{n}(B,C) = \Lambda_{n}(-(\alpha+\beta),\alpha\beta) = \Phi_{n}^{\mathrm{hom}}(\alpha,\beta).$$
	
	If $p \mid n$, writing $n = p^{a}n^{\prime}$ where $a = v_{p}(n)$ and $\mathrm{gcd}(n^{\prime},p)=1$, then 
	$$\Phi_{n}(X) = \dfrac{\Phi_{n^{\prime}}(X^{p^{a}})}{\Phi_{n^{\prime}}(X^{p^{a-1}})} = \Phi_{n^{\prime}}(X)^{p^{a}-p^{a-1}} = \Phi_{n^{\prime}}(X)^{\varphi(p^a)}.$$
	Homogenizing this identity yields
	$$\Phi_{n}^{\mathrm{hom}}(\alpha,\beta) = \Phi_{n^{\prime}}^{\mathrm{hom}}(\alpha,\beta)^{\varphi(p^a)}.$$
	Following from the argument in the last paragraph, the polynomial
	$$\Lambda_{n}(B,C) = \Lambda_{n^{\prime}}(B,C)^{\varphi(p^{a})} \in \mathbb{F}_{q}[B,C]$$
	satisfies
	$$\Lambda_{n}(B,C) = \Phi_{n}^{\mathrm{hom}}(\alpha,\beta).$$
	And the uniqueness is guaranteed by the fundamental theorem of symmetric polynomials.
\end{proof}
	
\begin{definition}
	For $n=1$ set $\Lambda_{1}(B,C)=1$, and for $n \geq 2$ define $\Lambda_{n}(B,C)$ to be the unique polynomial defined by Proposition \ref{prop 2}. The polynomial $\Lambda_{n}(B,C)$ is called the $n$-th Lucas atom.
\end{definition}
	
If $n$ is divisible by $p$, we write $n=p^{a}n^{\prime}$ where $a = v_{p}(n)$ and $\mathrm{gcd}(p,n^{\prime})=1$. The Lucas polynomials $\mathcal{U}_{n}(B,C)$, $\mathcal{V}_{n}(B,C)$, and the Lucas atom $\Lambda_{n}(B,C)$ can be obtained from $\mathcal{U}_{n^{\prime}}(B,C)$, $\mathcal{V}_{n^{\prime}}(B,C)$ and $\Lambda_{n^{\prime}}(B,C)$, respectively.

\begin{lemma}
	Let $n=p^{a}n^{\prime}$ with $a = v_{p}(n)$ and $\mathrm{gcd}(p,n^{\prime})=1$. Then we have
	\begin{enumerate}
		\item[(1)] $\Lambda_{n}(B,C) = \Lambda_{n^{\prime}}(B,C)^{\varphi(p^a)}$;
		\item[(2)] $\mathcal{V}_{n}(B,C) = \mathcal{V}_{n^{\prime}}(B,C)^{p^a}$; and
		\item[(3)] $\mathcal{U}_{n}(B,C) = 
		\begin{cases}
			(B^{2}-4C)^{\frac{p^a -1}{2}}\,\mathcal{U}_{n^{\prime}}(B,C)^{p^a}, \ \mathrm{if} \ p \ \mathrm{is} \ \mathrm{odd};\\
			B^{2^a -1}\,\mathcal{U}_{n^{\prime}}(B,C)^{2^a}, \ \mathrm{if} \ p=2.
		\end{cases}$
	\end{enumerate}
\end{lemma}

\begin{proof}
	The first conclusion has already been obtained by the proof of Proposition \ref{prop 2}. By the Binet formulas,
	$$\mathcal{V}_{n}(B,C) = \alpha^{n}+\beta^{n} = (\alpha^{n^{\prime}}+\beta^{n^{\prime}})^{p^a} = \mathcal{V}_{n^{\prime}}(B,C)^{p^a},$$
	and
	$$\mathcal{U}_{n}(B,C) = \dfrac{\alpha^{n}-\beta^{n}}{\alpha-\beta} = (\dfrac{\alpha^{n^{\prime}}-\beta^{n^{\prime}}}{\alpha-\beta})^{p^a}(\alpha-\beta)^{p^a -1} = (\alpha-\beta)^{p^a -1}\mathcal{U}_{n^{\prime}}(B,C)^{p^a}.$$
	Note that if $p$ is odd then
	$$(\alpha-\beta)^{p^a -1} = ((\alpha-\beta)^{2})^{\frac{p^a -1}{2}} = (B^2 -4C)^{\frac{p^a -1}{2}};$$
	while if $p=2$ then
	$$(\alpha-\beta)^{2^a -1} = (\alpha+\beta)^{2^a -1} = B^{2^a -1}.$$
	Hence we have 
	$$\mathcal{U}_{n}(B,C) = 
	\begin{cases}
		(B^{2}-4C)^{\frac{p^a -1}{2}}\,\mathcal{U}_{n^{\prime}}(B,C)^{p^a}, \ \mathrm{if} \ p \ \mathrm{is} \ \mathrm{odd};\\
		B^{2^a -1}\,\mathcal{U}_{n^{\prime}}(B,C)^{2^a}, \ \mathrm{if} \ p=2.
	\end{cases}$$
\end{proof}

Now we are ready to give the Lucas atomic decompositions of the Lucas polynomials $\mathcal{U}_{n}(B,C)$ and $\mathcal{V}_{n}(B,C)$.

\begin{theorem}\label{thm 1}
	For any $n \geq 1$, we have
	\begin{equation}\label{eq 1}
		\mathcal{U}_{n}(B,C) = \prod_{\substack{d\mid n\\d>1}}\Lambda_{d}(B,C)
	\end{equation}
	and
	\begin{equation}\label{eq 2}
		\mathcal{V}_{n}(B,C) = \prod_{\substack{d\mid 2n\\d\nmid n}}\Lambda_{d}(B,C).
	\end{equation}
    Here for $n=1$ the product $\prod\limits_{\substack{d\mid n\\d>1}}\Lambda_{d}(B,C)$ is set to be $1$.
\end{theorem}

\begin{proof}
	For any $n \geq 1$, as
	$$X^{n}-1 = \prod_{d\mid n}\Phi_{d}(X),$$
	then
	\begin{align*}
		\alpha^{n}-\beta^{n} &= \beta^{n}((\dfrac{\alpha}{\beta})^{n}-1)\\
		&= \beta^{n}\prod_{d\mid n}\Phi_{d}(\dfrac{\alpha}{\beta})\\
		&=\prod_{d\mid n}\Phi_{d}^{\mathrm{hom}}(\alpha,\beta)\\
		&=\prod_{d\mid n}\Lambda_{d}(B,C).
	\end{align*}
	Therefore by Lemma \ref{lem 3} and Corollary \ref{coro 1} we have
	$$\mathcal{U}_{n}(B,C) = \dfrac{\alpha^{n}-\beta^{n}}{\alpha-\beta} = \prod_{\substack{d\mid n\\d > 1}}\Lambda_{d}(B,C),$$
	and
	$$\mathcal{V}_{n}(B,C) = \dfrac{\mathcal{U}_{2n}(B,C)}{\mathcal{U}_{n}(B,C)} = \prod_{\substack{d\mid 2n\\d \nmid n}}\Lambda_{d}(B,C).$$
\end{proof}

\begin{remark}
	The factorizations \eqref{eq 1} and \eqref{eq 2} are called the Lucas atomic decomposition of $\mathcal{U}_{n}(B,C)$ and of $\mathcal{V}_{n}(B,C)$ respectively. The polynomials $\mathcal{U}_{n}(B,C)$, $\mathcal{V}_{n}(B,C)$ and $\Lambda_{n}(B,C)$ were first defined over $\mathbb{Q}$, in which case the Lucas atomic decompositions \eqref{eq 1} and \eqref{eq 2} are exactly the irreducible factorizations of $\mathcal{U}_{n}(B,C)$ and $\mathcal{V}_{n}(B,C)$ respectively. However, over finite fields they are not necessarily irreducible factorizations. This is quite analogous to the situation for the cyclotomic decomposition of $X^{n}-1$ over $\mathbb{Q}$ and over finite fields.
\end{remark}
	
By Theorem \ref{thm 1} we obtain the following inductive formula to compute Lucas atoms.

\begin{corollary}\label{coro 4}
	For any $n \geq 2$,
	$$\Lambda_{n}(B,C) = \dfrac{\mathcal{U}_{n}(B,C)}{\prod\limits_{\substack{d\mid n\\1 < d < n}}\Lambda_{d}(B,C)}.$$
\end{corollary}

For any $F(B,C) \in \mathbb{F}_{q}[B,C]$, we call its degree in the indeterminate $B$ over $\mathbb{F}_{q}[C]$ the $B$-degree of $F(B,C)$ and denote it by $\mathrm{deg}_{B}(F)$. We conclude this subsection by determining the leading terms of the Lucas polynomials and of the Lucas atoms with respect to the indeterminate $B$. In particular, their $B$-degrees are obtained.

\begin{proposition}
	\label{lem:B-degree-leading-coefficient}
	\begin{itemize}
		\item[(1)] For every $n\geq 1$, $\mathcal{U}_{n}(B,C)$ and $\mathcal{V}_{n}(B,C)$ can be expressed in the forms
		\[
		\mathcal{U}_n(B,C)
		=
		(-1)^{n-1}B^{n-1}
		+\text{terms of lower $B$-degree},
		\]
		and
		\[
		\mathcal{V}_n(B,C)
		=
		(-1)^nB^n
		+\text{terms of lower $B$-degree}.
		\]
		In particular, we have 
		\[
		\deg_B(\mathcal{U}_n)=n-1,
		\qquad
		\deg_B(\mathcal{V}_n)=n.
		\]
		\item[(2)] For every $n\geq 2$, $\Lambda_{n}(B,C)$ can be expressed in the forms
		\[
		\Lambda_n(B,C)
		=
		(-1)^{\varphi(n)}B^{\varphi(n)}
		+\text{terms of lower $B$-degree}.
		\]
		In particular, we have
		\[
		\deg_B(\Lambda_n)=\varphi(n).
		\]
	\end{itemize}
\end{proposition}

\begin{proof}
	For the conclusion $(1)$, we prove 
	\begin{equation}\label{eq 7}
		\mathcal{U}_n(B,C)
		=
		(-1)^{n-1}B^{n-1}
		+\text{terms of lower $B$-degree},
	\end{equation}
	and the proof of the latter identity for $\mathcal{V}_{n}(B,C)$ is similar. The identity \eqref{eq 7} holds clearly for $n=1$ and $n=2$. Suppose that $n\geq 3$ and that \eqref{eq 7} holds for
	$n-2$ and $n-1$. By the recurrence relation,
	\[
	\mathcal{U}_n(B,C)
	=
	-B\,\mathcal{U}_{n-1}(B,C)-C\,\mathcal{U}_{n-2}(B,C).
	\]
	Since
	\[
	-B\,\mathcal{U}_{n-1}(B,C)
	=
	(-1)^{n-1}B^{n-1}
	+\text{terms of lower $B$-degree},
	\]
	and
	\[
	\deg_B\bigl(C\,\mathcal{U}_{n-2}(B,C)\bigr)=n-3.
	\]
	Then $\mathcal{U}_{n}(B,C)$ can be written as
	\[
	\mathcal{U}_n(B,C)
	=
	(-1)^{n-1}B^{n-1}
	+\text{terms of lower $B$-degree}.
	\]
	In particular,
	\[
	\deg_B(\mathcal{U}_n)=n-1.
	\]
	
	Next we prove $(2)$ by induction on $n\geq 2$.
	By the Lucas atomic decomposition,
	\[
	\mathcal{U}_n(B,C)
	=
	\prod_{\substack{d\mid n\\d>1}}
	\Lambda_d(B,C).
	\]
	Since $\mathbb F_q[C][B]$ is an integral domain, both the
	$B$-degree and the $B$-leading coefficient are multiplicative
	under products. Therefore,
	\begin{equation}\label{eq 8}
		n-1
		=
		\deg_B(\mathcal{U}_n)
		=
		\sum_{\substack{d\mid n\\d>1}}
		\deg_B(\Lambda_d),
	\end{equation}
	and 
	\begin{equation}\label{eq 9}
		(-1)^{n-1}
		=
		\prod_{\substack{d\mid n\\d>1}}
		\operatorname{lc}_B(\Lambda_d),
	\end{equation}
	where $\operatorname{lc}_B(\Lambda_d)$ denotes the coefficient of the leading term of $\Lambda_{d}$ with respect to the indeterminate $B$.
	
	For $n=2$, since
	\[
	\mathcal{U}_2(B,C)=-B=\Lambda_2(B,C),
	\]
	we have
	\[
	\deg_B(\Lambda_2)=1=\varphi(2)
	\]
	and
	\[
	\operatorname{lc}_B(\Lambda_2)
	=
	-1
	=
	(-1)^{\varphi(2)}.
	\]
	
	Now suppose that $n>2$ and that, for every proper divisor
	$d>1$ of $n$,
	\[
	\deg_B(\Lambda_d)=\varphi(d)
	\]
	and
	\[
	\operatorname{lc}_B(\Lambda_d)
	=
	(-1)^{\varphi(d)}.
	\]
	By \eqref{eq 8} we obtain
	\[
	\begin{aligned}
		\deg_B(\Lambda_n)
		&=
		n-1
		-
		\sum_{\substack{d\mid n\\1<d<n}}
		\deg_B(\Lambda_d) \\
		&=
		n-1
		-
		\sum_{\substack{d\mid n\\1<d<n}}
		\varphi(d).
	\end{aligned}
	\]
	As $\sum\limits_{d\mid n}\varphi(d)=n$, then
	\[
	\sum_{\substack{d\mid n\\1<d<n}}
	\varphi(d)
	=
	n-1-\varphi(n),
	\]
	and consequently
	\[
	\deg_B(\Lambda_n)=\varphi(n).
	\]
	Moreover, the identity \eqref{eq 9} gives
	\[
	\begin{aligned}
		\operatorname{lc}_B(\Lambda_n)
		&=
		\frac{(-1)^{n-1}}
		{\displaystyle
			\prod_{\substack{d\mid n\\1<d<n}}
			(-1)^{\varphi(d)}} \\
		&=
		(-1)^{
			n-1-
			\sum\varphi(d)
		} \\
		&=
		(-1)^{\varphi(n)},
	\end{aligned}
	\]
	which implies
	\[
	\Lambda_n(B,C)
	=
	(-1)^{\varphi(n)}B^{\varphi(n)}
	+\text{terms of lower $B$-degree}.
	\]
\end{proof}

\subsection{The Lucas sequences associated to an irreducible quadratic polynomials}
\begin{definition}
	Let $f(X) = X^{2}+bX+c$ be an irreducible quadratic polynomial over $\mathbb{F}_{q}$. Associated to $f(X)$, we define two linear recurrence sequences $\{U_{n}\}_{n \geq 0}$ and $\{V_{n}\}_{n \geq 0}$ lying in $\mathbb{F}_{q}$ by
	\begin{itemize}
		\item[(1)] $U_{0}=0$, $U_{1}=1$, and
		$$U_{n} = -bU_{n-1}-cU_{n-2}$$
		for $n \geq 2$.
		\item[(2)] $V_{0}=2$, $V_{1}=-b$, and
		$$V_{n} = -bV_{n-1}-cV_{n-2}$$
		for $n \geq 2$.
	\end{itemize}
	The sequences $\{U_{n}\}_{n \geq 0}$ and $\{V_{n}\}_{n \geq 0}$ are called the Lucas sequences of the first type and of the second type, respectively, associated to $f(X)$.
\end{definition}

By definition for any $n \geq 0$ the terms $U_{n}$ and $V_{n}$ can be realized as the evaluations 
\begin{equation}\label{eq 3}
	U_{n} = \mathcal{U}_{n}(b,c), \quad V_{n} = \mathcal{V}_{n}(B,C)
\end{equation}
of $\mathcal{U}_{n}(B,C)$ and of $\mathcal{V}_{n}(B,C)$, respectively, at the point $(b,c)$. Following from Lemma \ref{lem 3} the terms $U_{n}$ and $V_{n}$ can also be obtained by the Binet formulas

\begin{lemma}\label{lem 5}
	Let $\theta$ be a root of $f(X)$. Then for every $n \geq 0$, 
	\begin{equation}\label{eq 4}
		U_{n} = \dfrac{\theta^{n}-\theta^{nq}}{\theta-\theta^{q}}, \quad V_{n} = \theta^{n}+\theta^{nq}.
	\end{equation}
\end{lemma}

\begin{proof}
	Since $f(X)$ is irreducible over $\mathbb{F}_{q}$, its two roots are $\theta$ and $\theta^{q}$, which are distinct. Therefore the quotient $\frac{\theta^{n}-\theta^{nq}}{\theta-\theta}$ is well-defined. Evaluation the identities
	$$\mathcal{U}_{n}(B,C) = \dfrac{\alpha^{n}-\beta^{n}}{\alpha-\beta}, \quad \mathcal{V}_{n}(B,C) = \alpha^{n}+\beta^{n}$$
	at $(\alpha,\beta) = (\theta,\theta^{q})$ yields
	$$U_{n} = \dfrac{\theta^{n}-\theta^{nq}}{\theta-\theta^{q}}, \quad V_{n} = \theta^{n}+\theta^{nq}.$$
\end{proof}

\begin{corollary}\label{coro 2}
	For any $n \geq 1$, we have
	$$U_{2n} = U_{n}V_{n}, \quad V_{2n} = V_{n}^{2}-2c^{n}.$$
\end{corollary}

\begin{proof}
	The conclusions follow from evaluating the identities in Corollary \ref{coro 1} at point $(\alpha,\beta) = (\theta,\theta^{q})$.
\end{proof}

A reason for introducing these two sequence is that $\{U_{n}\}_{n \geq 0}$ gives rise to the reminders of $X^{n}$ modulo $f(X)$ for all $n \geq 1$, which is stated explicitly as follows.

\begin{proposition}
	For every \(n\ge1\),
	\begin{equation*}
		X^n\equiv U_n X-cU_{n-1}\pmod{f(X)}.
	\end{equation*}
\end{proposition}

\begin{proof}
	The formula is clear for \(n=1\).  If it holds for \(n\), then
	\[
	\begin{aligned}
		X^{n+1}
		&\equiv U_nX^2-cU_{n-1}X\\
		&\equiv U_n(-bX-c)-cU_{n-1}X\\
		&=(-bU_n-cU_{n-1})X-cU_n\\
		&=U_{n+1}X-cU_n\pmod{f(X)},
	\end{aligned}
	\]
	which proves the result by induction.
\end{proof}

\begin{corollary}\label{coro 3}
	Assume that $r \geq 1$ is the smallest positive integer such that $U_{r}=0$. Then the binomial order of $f(X)$ is $r$ and the minimal binomial multiple of $f(X)$ is given by
	$$X^{r} + cU_{r-1}.$$
\end{corollary}

\begin{remark}
	By the identities in \eqref{eq 3}, the binomial order of $f(X)$ can be identified as the smallest positive integer $r$ such that $(b,c)$ is a root of the Lucas polynomial $\mathcal{U}_{n}(B,C)$.
\end{remark}

\section{Preliminary results}
This section collects some preliminary results needed for developing the criterion for primitive quadratic polynomials over $\mathbb{F}_{q}$. In the first subsection, fixing the constant term, we count the irreducible quadratic polynomials and the primitive quadratic polynomials over $\mathbb{F}_{q}$ respectively. And in the second subsection, for an irreducible polynomial $f(X) = X^{2}+bX+c \in \mathbb{F}_{q}[X]$, where $c$ is a primitive element of $\mathbb{F}_{q}$, we analyze the possible values of its binomial order.

\subsection{Enumerations of irreducible and of primitive quadratic polynomials with a fixed constant term}
In this section, unless stating otherwise, $c$ is assumed to be a primitive element of $\mathbb{F}_{q}$. For such a $c$ and an element $b$ of $\mathbb{F}_{q}$, we denote by $f_{b,c}(X)$ the polynomial
$$f_{b,c}(X) = X^{2}+bX+c.$$
Fixing $c$, define
$$N_{\mathrm{irr}}(q,c) = \#\{b\in\mathbb{F}_{q} \ | \ f_{b,c}\text{ is irreducible}\}$$
and
$$N_{\mathrm{prim}}(q,c) = \#\{b\in\mathbb{F}_{q} \ | \ f_{b,c}\text{ is primitive}\}.$$
We compute $N_{\mathrm{irr}}(q,c)$ and $N_{\mathrm{prim}}(q,c)$ respectively in the following two lemmas.
	
\begin{lemma}\label{lem 6}
	Given a primitive element $c \in \mathbb{F}_{q}$, we have
	$$
	N_{\mathrm{irr}}(q,c)=
	\begin{cases}
		\dfrac{q}{2},& \text{if} \ q\ \text{is even};\\[2mm]
		\dfrac{q+1}{2},& \text{if} \ q\ \text{is odd}.
	\end{cases}
	$$
	\end{lemma}	
	
\begin{proof}
	For any $b \in \mathbb{F}_{q}$, the polynomial $f_{b,c}(X)$ is irreducible if and only if there exists $\theta \in \mathbb{F}_{q^2} \setminus \mathbb{F}_{q}$ such that
	$$f_{b,c}(X) = (X-\theta)(X-\theta^{q}),$$
	which amounts to that $b = -\theta - \theta^{q} = -\mathrm{Tr}_{q^{2}/q}(\theta)$ and $c = \theta^{q+1} = \mathrm{N}_{q^{2}/q}(\theta)$. Thus $N_{\mathrm{irr}}(q,c)$ is equal to the number of Frobenius orbits $\{\theta,\theta^{q}\} \subseteq \mathbb{F}_{q^2} \setminus \mathbb{F}_{q}$ satisfying $\mathrm{N}_{q^{2}/q}(\theta) = c$.
	
	Consider the set
	$$S_{c} = \{\theta \in \mathbb{F}_{q^2}^{\ast} \ | \ \mathrm{N}_{q^{2}/q}(\theta) = c\}.$$
	As the norm map $\mathrm{N}_{q^{2}/q}: \mathbb{F}_{q^2}^{\ast} \rightarrow \mathbb{F}_{q}^{\ast}$ is surjective and its kernel has order $q+1$, then $\# S_{c} = q+1$.
	
	Suppose that $q$ is odd. If $\theta \in S_{c} \cap \mathbb{F}_{q}$, then
	$$c = \theta^{q+1} = \theta^{2},$$
	which contradicts the fact that $c$ is primitive in $\mathbb{F}_{q}$. Hence $S_{c} \subseteq \mathbb{F}_{q^2} \setminus \mathbb{F}_{q}$, and consequently can be partitioned into $\frac{q+1}{2}$ Frobenius orbits. It follows immediately that $$N_{\mathrm{irr}}(q,c) = \frac{q+1}{2}.$$
	
	Next suppose that $q$ is even. The squaring map on $\mathbb{F}_{q}$ is a
	bijection, so there is exactly one $d \in \mathbb{F}_{q}$ with $d^2=c$.
	It is the unique element of $S_c\cap\mathbb{F}_{q}$.  The remaining $q$
	elements form $\frac{q}{2}$ two-element Frobenius orbits. Thus one has $$N_{\mathrm{irr}}(q,c) = \frac{q}{2}.$$
\end{proof}
	
\begin{lemma}\label{lem 7}
	Given a primitive element $c \in \mathbb{F}_{q}$, we have
	$$
	N_{\mathrm{prim}}(q,c)=
	\begin{cases}
		\dfrac{\varphi(q+1)}{2},& \text{if} \ q\ \text{is even};\\
		\varphi(q+1),& \text{if} \ q\ \text{is odd},
	\end{cases}
	$$
	where $\varphi$ is Euler's totient function.
\end{lemma}
	
\begin{proof}
	Since $\mathbb{F}_{q^2}^{\ast}$ is a cyclic group of order $q^2-1$, one can find a generator of $\mathbb{F}_{q^2}^{\ast}$, i.e., a primitive element $\gamma$ of $\mathbb{F}_{q^2}$ such that $c = \gamma^{q+1}$. For any $0 \leq e < q^{2}-1$, $\gamma^{e} \in S_{c}$ if and only if
	$$\gamma^{e(q+1)} = c = \gamma^{q+1},$$
	which amounts to $e \equiv 1 \pmod{q-1}$. Therefore we have
	$$S_{c} = \{\gamma^{1+k(q-1)} \ | \ 0 \leq k < q+1\}.$$
	Any element $\gamma^{1+k(q-1)}$ in $S_{c}$ is primitive in $\mathbb{F}_{q^2}$ if and only if $\mathrm{gcd}(1+k(q-1),q^{2}-1)=1$, which is equivalent to $\mathrm{gcd}(1+k(q-1),q+1)=1$ as $1+k(q-1)$ is automatically coprime to $q-1$. Note
	$$1+k(q-1) \equiv 1-2k \pmod{q+1}.$$
	
	If $q$ is even, the map
	$$\mathbb{Z}/(q+1)\mathbb{Z} \rightarrow \mathbb{Z}/(q+1)\mathbb{Z}; \ k \mapsto 1-2k$$
	is a bijection. Therefore $S_{c}$ contains $\varphi(q+1)$ primitive element of $\mathbb{F}_{q^{2}}$, which form $\frac{\varphi(q+1)}{2}$ Frobenius orbits. Hence
	$$N_{\mathrm{prim}}(q,c) = \frac{\varphi(q+1)}{2}.$$
	
	If $q$ is odd, then
	$$\mathbb{Z}/(q+1)\mathbb{Z} \rightarrow \mathbb{Z}/(q+1)\mathbb{Z}; \ k \mapsto 1-2k$$
	is a $2$-to-$1$ map with image $1+2\mathbb{Z}/(q+1)\mathbb{Z}$. Since $q+1$ is even, the image of any integer coprime to $q+1$ lies in $1+2\mathbb{Z}/(q+1)\mathbb{Z}$. Thus there are in total $2\varphi(q+1)$ primitive elements of $\mathbb{F}_{q^2}$ lying in $S_{c}$, which form $\varphi(q+1)$ Frobenius orbits. Consequently we have 
	$$N_{\mathrm{prim}}(q,c) = \varphi(q+1).$$
\end{proof}
	
\subsection{The binomial order of an irreducible quadratic polynomial}
\begin{lemma}\label{lem 8}
	Let $G$ be a cyclic group of order $n$. Let $g \in G$, and $t$ be a positive integer. If $\mathrm{ord}(g^{t}) = n$, then $\mathrm{ord}(g) = n$ and $\mathrm{gcd}(n,t)=1$. 
\end{lemma}

\begin{proof}
	Since $\mathrm{ord}(g^{t}) \mid \mathrm{ord}(g)$ and $\mathrm{ord}(g) \mid n$, then $\mathrm{ord}(g^{t}) = n$ indicates $\mathrm{ord}(g) = n$. And by
	$$n = \mathrm{ord}(g^{t}) = \dfrac{n}{\mathrm{gcd}(n,t)},$$
	we have $\mathrm{gcd}(n,t)=1$.
\end{proof}

\begin{lemma}\label{lem 9}
	Let $f(X) = X^{2}+bX+c$ be an irreducible polynomial over $\mathbb{F}_{q}$, with minimal binomial multiple $X^{r}-\lambda$. Then $r \mid q+1$ and $c = \lambda^{\frac{q+1}{r}}$. 
\end{lemma}

\begin{proof}
	Assume that $\theta$ is a root of $f(X)$ lying in $\mathbb{F}_{q^2}$. Then the binomial order $r$ of $f(X)$ is equal to the order of the coset $\theta\mathbb{F}_{q}^{\ast}$ in the quotient group $\mathbb{F}_{q^2}^{\ast}/\mathbb{F}_{q}^{\ast}$. As $\theta^{q+1} =c \in \mathbb{F}_{q}^{\ast}$ then $r \mid q+1$. The latter assertion follows from
	$$c = \theta^{q+1} = (\theta^{r})^{\frac{q+1}{r}} = \lambda^{\frac{q+1}{r}}.$$
\end{proof}

Assume that 
$$q+1 = 2^{m}p_{1}^{e_{1}}\cdots p_{s}^{e_{s}},$$
where $p_{1},\cdots,p_{s}$ are pairwise distinct odd primes, $m > 0$, and $e_{1},\cdots,e_{s}$ are positive integers. Let $f(X) = X^{2}+bX+c$ be an irreducible polynomial over $\mathbb{F}_{q}$, where $c$ is a primitive element of $\mathbb{F}_{q}$. The possible values of the binomial order of $f(X)$ are given by the proposition below.

\begin{proposition}\label{prop 1}
	With the notations given as above, the binomial order of $f(X)$ is in the form
	$$\mathrm{ord}_{\mathrm{b}}(f) = 2^{m}p_{1}^{j_{1}}\cdots p_{s}^{j_{s}},$$
	where $0 \leq j_{i} \leq e_{i}$ for $i=1,\cdots,s$. Furthermore, if $c$ is a primitive element of $\mathbb{F}_{q}$, then $f(X)$ is primitive if and only if $j_{i}=e_{i}$ for every $1\leq i\leq s$.
\end{proposition}

\begin{proof}
	Let $r = \mathrm{ord}_{\mathrm{b}}(f)$, and let $X^{r}-\lambda$ be the minimal binomial multiple of $f(X)$. By Lemma \ref{lem 9} we obtain $\mathrm{ord}(\lambda)=q-1$ and $\mathrm{gcd}(\frac{q+1}{r},q-1)=1$.
	Note that 
	$$
	\mathrm{gcd}(q+1,q-1)= \mathrm{gcd}(q+1,2)
	\begin{cases}
		2,& \text{if} \ q\ \text{is odd};\\
		1,& \text{if} \ q\ \text{is even},
	\end{cases}
	$$
	thus $\mathrm{gcd}(\frac{q+1}{r},q-1)=1$ if and only if $v_{2}(r) = v_{2}(q+1)$, that is, $r$ is of the form 
	$$r = 2^{m}p_{1}^{j_{1}}\cdots p_{s}^{j_{s}},$$
	where $0 \leq j_{i} \leq e_{i}$ for all $i=1,\cdots,s$. The last assertion follows immediately from the above argument and Lemma \ref{lem 4}.
\end{proof}

\section{The optimal coefficients-based criterion for primitive quadratic polynomials over finite fields}\label{sec: criterion}
The purpose of this section is to develop a criterion for primitive quadratic polynomials over finite fields, which only depends on their coefficients. We will prove that there exists a polynomial
$$P_{q}(B,C) \in \mathbb{F}_{q}[B,C]$$
such that any quadratic polynomial $f(X) = X^{2}+bX+c$ over $\mathbb{F}_{q}$, where $c$ is a primitive element of $\mathbb{F}_{q}$, is primitive if and only if $(b,c)$ is a root of $P_{q}(B,C)$. Further, we require that $P_{q}(B,C)$ satisfies the property that, setting $C=c$ for a fixed primitive element $c$ of $\mathbb{F}_{q}$, the polynomial $P_{q}(B,c) \in \mathbb{F}_{q}[B]$ is monic and squarefree, and its roots are in an one-to-one correspondence with the coefficients $b$ for which $X^{2}+bX+c$ is primitive. Such a polynomial is called an optimal determining polynomial for primitive quadratic polynomials over $\mathbb{F}_{q}$.

In the following two subsections, we will construct $P_{q}(B,C)$ separately in the cases where $q$ is odd and where $q$ is even. And in the last subsection we will discuss the uniqueness property of $P_{q}(B,C)$.

\subsection{Over a finite field of odd characteristic}
In this subsection, we assume that $q$ is odd and write $N=q+1$. Further, assume that the prime factorization of $N$ is given by
$$N = 2^{m}p_{1}^{e_{1}}\cdots p_{s}^{e_{s}},$$
where $p_{1},\cdots,p_{s}$ are pairwise distinct odd primes, and $m,e_{1},\cdots,e_{s}$ are positive integers. We denote the maximal odd divisor $p_{1}^{e_{1}}\cdots p_{s}^{e_{s}}$ of $N$ by $R$.

We adopt the following notations from Section \ref{sec Luc}. Let $B$ and $C$ be algebraically independent indeterminates over $\mathbb{F}_{q}$. For each integer $n \geq 1$, denote by $\mathcal{U}_{n}(B,C)$, $\mathcal{V}_{n}(B,C)$ and $\Lambda_{n}(B,C)$ the $n$-th Lucas polynomials of the first type, the $n$-th Lucas polynomials of the second type and the $n$-th Lucas atoms respectively.

\begin{lemma}\label{lem 10}
	Let $f(X) = X^{2}+bX+c$ be a quadratic polynomial over $\mathbb{F}_{q}$ with $c \neq 0$, and let $\theta_{1}$ and $\theta_{2}$ be the two roots of $f(X)$. Then for any positive integer $n$ that is coprime to $q$, $\Lambda_{n}(b,c)=0$ if and only if $\mathrm{ord}(\frac{\theta_{1}}{\theta_{2}})=n$.
\end{lemma}

\begin{proof}
	By the definition of the Lucas atom $\Lambda_{n}$ one has
	$$\Lambda_{n}(b,c) = \theta_{2}^{\varphi(n)}\Phi_{n}(\frac{\theta_{1}}{\theta_{2}}).$$
	As $c \neq 0$, neither $\theta_{1}$ nor $\theta_{2}$ equals to $0$. Note that $n$ is coprime to $q$, then $\Lambda_{n}(b,c)=0$ if and only if $\Phi_{n}(\frac{\theta_{1}}{\theta_{2}})=0$, which amounts to that $\frac{\theta_{1}}{\theta_{2}}$ is a primitive $n$-th root of unity.
\end{proof}

\begin{remark}
	Notice that the proof of Lemma \ref{lem 10} does not rely on the fact that $q$ is odd, thus Lemma \ref{lem 10} in fact holds for any finite field $\mathbb{F}_{q}$.
\end{remark}

\begin{lemma}
	Let $n=2^{m}r$ where $r \mid R$. For any $b \in \mathbb{F}_{q}$ and any primitive element $c$ of $\mathbb{F}_{q}$, if $\Lambda_{n}(b,c)=0$ then $f(X) = X^{2}+bX+c$ is irreducible over $\mathbb{F}_{q}$. 
\end{lemma}

\begin{proof}
	Suppose that $f(X)$ is reducible, then it has two roots, say $\theta_{1}$ and $\theta_{2}$, in $\mathbb{F}_{q}$. By Lemma \ref{lem 10} the order of $\frac{\theta_{1}}{\theta_{2}}$ is $n$. Since $\frac{\theta_{1}}{\theta_{2}} \in \mathbb{F}_{q}$ then $n \mid q-1$. By assumption $n \mid q+1$, thus one has
	$$2^{m}r = n \mid \mathrm{gcd}(q-1,q+1)=2,$$
	which indicates that $m=1$ and $r=1$. As $q$ is odd, the only primitive $2$-th root of unity in $\mathbb{F}_{q}$ is $-1$, which implies that $\frac{\theta_{1}}{\theta_{2}}=-1$. It follows that 
	$$b=-(\theta_{1}+\theta_{2}) = 0, \quad c = \theta_{1}\theta_{2} = -\theta_{1}^{2}.$$
	Note that $v_{2}(N)=m=1$, therefore $q \equiv 1 \pmod{4}$, indicating that $-1$ is a quadratic residue in $\mathbb{F}_{q}$, and consequently so is $-\theta_{1}^{2}=c$. This is a contradiction with the fact that $c$ is primitive in $\mathbb{F}_{q}$.
 \end{proof}
 
 \begin{theorem}\label{thm 2}
 	Let 
 	$$I_{q}(B,C) = (-1)^{\frac{N}{2}}\mathcal{V}_{\frac{N}{2}}(B,C)$$
 	and
 	$$P_{q}(B,C) = \Lambda_{N}(B,C).$$
 	For any primitive element $c$ of $\mathbb{F}_{q}$, we have
 	\begin{itemize}
 		\item[(1)] $I_{q}(B,c) \in \mathbb{F}_{q}[B]$ has $\frac{N}{2}$ roots in $\mathbb{F}_{q}$, which are exactly the coefficients $b$ for which $X^{2}+bX+c$ is irreducible over $\mathbb{F}_{q}$;
 		\item[(2)] $P_{q}(B,c) \in \mathbb{F}_{q}[B]$ has $\varphi(N)$ roots in $\mathbb{F}_{q}$, which are exactly the coefficients $b$ for which $X^{2}+bX+c$ is primitive. 
 	\end{itemize}
 	Therefore, the polynomial $P_{q}(B,C)$ is an optimal determining polynomial for primitive quadratic polynomial over $\mathbb{F}_{q}$.
 \end{theorem}
 
 \begin{proof}
 	As $\frac{N}{2} = 2^{m-1}R$, a divisor $d$ of $N$ does not divide $\frac{N}{2}$ if and only if it has the form $d = 2^{m}r$ where $r \mid R$. Then the atomic decomposition of $\mathcal{V}_{\frac{N}{2}}(B,C)$ gives
 	$$\mathcal{V}_{\frac{N}{2}}(B,C) = \prod_{\substack{d\mid N\\d \nmid \frac{N}{2}}}\Lambda_{d}(B,C) = \prod_{r \mid R}\Lambda_{2^{m}r}(B,C).$$
 	Fix a primitive element $c$ of $\mathbb{F}_{q}$. For any $b \in \mathbb{F}_{q}$, if $f_{b}(X) = X^{2}+bX+c$ is irreducible, by Proposition \ref{prop 1}
 	$$\mathrm{ord}(\theta^{q-1}) = \mathrm{ord}_{\mathrm{b}}(f_{b}) = 2^{m}r,$$
 	where $\theta$ is a root of $f_{b}(X)$. It follows from Lemma \ref{lem 10} that $(b,c)$ is a root of $\Lambda_{2^{m}r}(B,C)$, and hence a root of $\mathcal{V}_{\frac{N}{2}}(B,C)$. This proves that each $b \in \mathbb{F}_{q}$ such that $f_{b}(X)$ is irreducible gives rise to a root of $\mathcal{V}_{\frac{N}{2}}(B,c)$. Notice that there are $\frac{N}{2}$ such elements $b \in \mathbb{F}_{q}$, as many as the possible roots of $\mathcal{V}_{\frac{N}{2}}(B,c)$. Therefore the assertion $(1)$ holds.
 	
 	Let $b \in \mathbb{F}_{q}$ be a root of $\Lambda_{N}(B,c)$. As $\Lambda_{N}(B,c) \mid \mathcal{V}_{\frac{N}{2}}(B,c)$, $b$ is a root of $\mathcal{V}_{\frac{N}{2}}(B,c)$. By the above argument, the polynomial $f_{b}(X) = X^{2}+bX+c$ is irreducible, and its roots are in the form $\theta$ and $\theta^{q}$. Following from Lemma \ref{lem 10} again, one has 
 	$$\mathrm{ord}(\theta^{q-1}) = \mathrm{ord}_{\mathrm{b}}(f_{b}) = N,$$
 	which indicates by Proposition \ref{prop 1} that $f_{b}(X)$ is primitive. Since the roots of $V_{\frac{N}{2}}(B,c)$ are all in $\mathbb{F}_{q}$ and pairwise distinct, then so are the roots of $\Lambda_{N}(B,c)$. Note that there are $\varphi(N)$ roots of $\Lambda_{N}(B,c)$, as many as the elements $b \in \mathbb{F}_{q}$ for which $f_{b}(X)$ is primitive. Hence the assertion $(2)$ holds. As $N =q+1 \geq 4$, $\varphi(N)$ is even. Then by Proposition \ref{lem:B-degree-leading-coefficient} the coefficient of the leading term of $P_{q}(B,C)=\Lambda_{N}(B,C)$ with respect to $B$ is $1$. It follows that $P_{q}(B,C)$ is an optimal determining polynomial for primitive quadratic polynomial over $\mathbb{F}_{q}$.
 \end{proof}
	
Theorem \ref{thm 2} shows that when the indeterminate $C$ is specialized to primitive elements of $\mathbb{F}_{q}$, the solutions of 
$$I_{q}(B,C)=0$$
completely determine the irreducible quadratic polynomials over $\mathbb{F}_{q}$, and the solutions of
$$P_{q}(B,C)=0$$
completely determine the primitive quadratic polynomials over $\mathbb{F}_{q}$. However, computing the polynomial $P_{q}(B,C) = \Lambda_{N}(B,C)$ via the definition of the Lucas atom $\Lambda$ or via the recurrence formula given in Corollary \ref{coro 4} is complicated. We next
give a criterion involving only Lucas polynomials, which can be
computed directly from their recurrence relations.

\begin{lemma}
	\label{lem:coprime-Lucas-atoms}
	Let $d,e\geq 2$ be distinct positive integers such that
	\[
	p\nmid de.
	\]
	Then
	\[
	\gcd(\Lambda_d(B,C),\Lambda_e(B,C))=1.
	\]
\end{lemma}

\begin{proof}
	Since the roots of $\Phi_d(X)$ are the primitive
	$d$-th roots of unity, and the roots of $\Phi_e(X)$
	are the primitive $e$-th roots of unity, then $d\neq e$ implies that $\Phi_d(X)$ and $\Phi_e(X)$ have no common root. Thus 
	\[
	\gcd(\Phi_d(X),\Phi_e(X))=1,
	\]
	and consequently 
	\[
	\gcd(
	\Phi_d^{\mathrm{hom}}(\alpha,\beta),
	\Phi_e^{\mathrm{hom}}(\alpha,\beta)
	)=1.
	\]
    Note that $B = -(\alpha,\beta)$ and $C = \alpha\beta$. By the fundamental theorem of symmetric polynomials, the symmetric polynomials in $\mathbb{F}_{q}[\alpha,\beta]$ are in an one-to-one correspondence with the polynomials in $\mathbb{F}_{q}[B,C]$. Hence we have
    \[
    \gcd(\Lambda_d(B,C),\Lambda_e(B,C))= \gcd(
    \Phi_d^{\mathrm{hom}}(\alpha,\beta),
    \Phi_e^{\mathrm{hom}}(\alpha,\beta))=1.
    \]
    \end{proof}
    
    For any polynomials $F(X), G(X) \in \mathbb F_q[C][B]$, we denote by
    $\operatorname{lcm}_B(F(X),G(X))$ their least common multiple normalized to
    be monic with respect to the indeterminate $B$.
    
    \begin{proposition}
    \label{prop:Lucas-polynomial-criterion}
    Let the notation be as in Theorem \ref{thm 2}. Define
    \[
    E_q(B,C)
    =
    \operatorname{lcm}
    (\mathcal{V}_{\frac{N}{2p_{1}}}(B,C),\ldots,
    \mathcal{V}_{\frac{N}{2p_{s}}}(B,C)).
    \]
    Then an irreducible quadratic polynomial $f(X)=X^2+bX+c\in\mathbb F_q[X]$, where $c$ is a primitive element of $\mathbb F_q$, is primitive
    if and only if $E_q(b,c)\neq 0$.
    \end{proposition}
    
    \begin{proof}
    For each $1\leq i\leq s$, put
    \[
    n_i=\frac{N}{2p_i}
    =2^{m-1}\frac{R}{p_i}.
    \]
    From Theorem \ref{thm 2}, it suffices to show $E_{q}(B,C)$ and $\frac{I_{q}(B,C)}{P_{q}(B,C)}$ differ by a scalar multiplication by an element in $\mathbb{F}_{q}^{\ast}$.
    
    For any $1 \leq i \leq s$,
    \[
    \mathcal{V}_{n_i}(B,C)
    =
    \prod_{\substack{d\mid 2n_i\\d\nmid n_i}}
    \Lambda_d(B,C) = \prod_{r \mid \frac{R}{p_{i}}}\Lambda_{2^m r}(B,C).
    \]
    Setting
    \[
    \Omega
    =
    \bigcup_{i=1}^s
    \left\{
    r:r\mid\frac{R}{p_i}
    \right\},
    \]
    we claim that 
    \[
    E_q(B,C)
    =
    \prod_{r\in\Omega}
    \Lambda_{2^mr}(B,C).
    \]
    
    Clearly each $\mathcal{V}_{n_{i}}$ divides $\prod\limits_{r \in \Omega}\Lambda_{2^m r}$, therefore so does $\operatorname{lcm}
    (\mathcal{V}_{n_{1}},\ldots,
    \mathcal{V}_{n_{s}})$. On the other hand, for each $r \in \Omega$, $\Lambda_{2^m r}$ is a factor of some $\mathcal{V}_{n_{i}}$, and thus is a factor of $\operatorname{lcm}
    (\mathcal{V}_{n_{1}},\ldots,
    \mathcal{V}_{n_{s}})$. And by Lemma \ref{lem:coprime-Lucas-atoms}, all the $\Lambda_{2^m r}$, $r \in \Omega$, are pairwise coprime. Hence 
    \[
    \prod_{r\in\Omega}
    \Lambda_{2^mr} \mid \operatorname{lcm}
    (\mathcal{V}_{n_{1}},\ldots,
    \mathcal{V}_{n_{s}}),
    \]
    proving the claim.
    
    Notice that $r \in \Omega$ ranges over all the positive divisors of $R$ except $R$ itself. Thus we have
    $$\operatorname{lcm}
    (\mathcal{V}_{n_{1}},\ldots,
    \mathcal{V}_{n_{s}}) = \dfrac{\prod\limits_{r \mid R}\Lambda_{2^m r}}{\Lambda_{2^m R}} = \dfrac{\mathcal{V}_{\frac{N}{2}}}{\Lambda_{N}},$$
    which by Proposition \ref{lem:B-degree-leading-coefficient} differs from $\frac{I_{q}}{P_{q}}$ by $\pm1$.
    \end{proof}
    
\begin{examples}
	Let $q$ be a power of odd prime such that 
	$$q+1 = N = 2^m\pi,$$
	where $m \geq 1$ and $\pi$ is an odd prime. Proposition \ref{prop:Lucas-polynomial-criterion} gives that for any pair $(b,c)$, where $b \in \mathbb{F}_{q}$ and $c$ is a primitive element of $\mathbb{F}_{q}$, the polynomial
	$$X^{2}+bX+c \in \mathbb{F}_{q}[X]$$
	is primitive if and only if 
	$$\mathcal{V}_{2^{m-1}}(b,c) \neq 0.$$
	The next table lists the concrete coefficients-based criteria for $X^{2}+bX+c$ to be primitive in the cases $m=2$, $4$, $8$ and $16$ respectively.
	
	\begin{table}[H]
		\centering
		\label{Table2}
		\begin{tabular}{cc}
			\toprule
			$q+1$ & Exceptional equation\\
			\midrule
			$2\pi$ &
			$b\neq0$\\[1ex]
			
			$4\pi$ &
			$b^2\neq2c$\\[1ex]
			
			$8\pi$ &
			$b^4-4b^2c+2c^2\neq0$\\[1ex]
			
			$16\pi$ &
			$b^8-8b^6c+20b^4c^2-16b^2c^3+2c^4\neq0$\\
			\bottomrule
		\end{tabular}
	\end{table} 
\end{examples}

Observe that in the cases where $N=2\pi$ and where $N=4\pi$, the equations $b \neq 0$ and $b^{2} \neq 2c$ precisely give rise to the criteria given by Vega (\cite{VegaCharacterization}, \cite{VegaNecessary}) for these two families of parameters. We are thus led to an interpretation of Vega's criteria: they are not isolated examples, but rather come from certain Lucas polynomial family, which capture complete information on primitive quadratic polynomials over finite fields.

\subsection{Over a finite field of characteristic $2$}
In this subsection, we turn to the case where $q$ is a power of $2$. We set $N=q+1$ and assume
\[
N=p_1^{e_1}\cdots p_s^{e_s},
\]
where $p_1,\ldots,p_s$ are pairwise distinct odd primes and
$e_1,\ldots,e_s$ are positive integers. The main difference from the odd-characteristic case is that the relevant Lucas atom occurs with multiplicity $2$ in the indeterminate $B$. Therefore to obtain an optimal determining polynomial for primitive quadratic polynomials over $\mathbb{F}_{q}$, a map called the inverse-Frobenius operator is introduced to remove the multiplicity.

\begin{lemma}
	\label{lem:char2-even-B}
	For every odd integer $n>1$, one has
	\[
	\mathcal{U}_n(B,C),\ \Lambda_n(B,C)\in\mathbb F_q[B^2,C].
	\]
\end{lemma}

\begin{proof}
	We first prove the assertion for $U_n(B,C)$.
	For every integer $r\geq 0$, we claim that
	\[
	\mathcal{U}_{2r+1}(B,C)
	=
	\mathcal{U}_{r+1}(B,C)^2+C\,\mathcal{U}_r(B,C)^2,
	\]
	which clearly implies that $\mathcal{U}_n(B,C) \in \mathbb{F}_q[B^2,C]$. Let $\alpha$ and $\beta$ be the roots of
	\[
	T^2+BT+C \in \mathbf{k},
	\]
	where $\mathbf{k} = \mathbb{F}_{q}(B,C)$. The Binet formula gives
	\[
	\mathcal{U}_r(B,C)
	=
	\frac{\alpha^r-\beta^r}{\alpha-\beta},
	\]
	therefore one has
	\[
	\begin{aligned}
		\mathcal{U}_{r+1}^2+C\,\mathcal{U}_r^2
		&=
		\frac{
			(\alpha^{r+1}-\beta^{r+1})^2
			-\alpha\beta(\alpha^r-\beta^r)^2
		}{(\alpha-\beta)^2} \\
		&=
		\frac{
			\alpha^{2r+2}+\beta^{2r+2}
			-\alpha^{2r+1}\beta
			-\alpha\beta^{2r+1}
		}{(\alpha-\beta)^2} \\
		&=
		\frac{
			(\alpha-\beta)
			(\alpha^{2r+1}-\beta^{2r+1})
		}{(\alpha-\beta)^2} \\
		&=
		\mathcal{U}_{2r+1}(B,C),
	\end{aligned}
	\]
	which proves the claim.
	
	Next we prove the assertion for the Lucas atom $\Lambda_{n}(B,C)$.
	Let $n>1$ be odd. By the Lucas atomic decomposition,
	\begin{equation}\label{eq 10}
		\mathcal{U}_n(B,C)
		=
		\prod_{\substack{d\mid n\\d>1}}
		\Lambda_d(B,C).
	\end{equation}
	Since $n$ is odd and the characteristic is $2$, every divisor
	$d$ of $n$ is coprime to the characteristic. Hence, by
	Lemma~\ref{lem:coprime-Lucas-atoms}, the distinct Lucas atoms $\Lambda_d(B,C)$, for $d\mid n$ and $d>1$, are pairwise coprime.
	
	Since
	\[
	\mathcal{U}_n(B,C)\in\mathbb F_q[B^2,C],
	\]
	the formal partial derivative of $\mathcal{U}_{n}(B,C)$ with respect to $B$ is vanishing, i.e.,
	\[
	\frac{\partial\,\mathcal{U}_n}{\partial B}=0.
	\]
	Differentiating \eqref{eq 10} with respect to $B$ gives
	\[
	0
	=
	\sum_{\substack{d\mid n\\d>1}}
	\frac{\partial\Lambda_d}{\partial B}
	\prod_{\substack{e\mid n\\e>1\\e\neq d}}
	\Lambda_e,
	\]
	which modulo a fixed
	$\Lambda_d(B,C)$ yields
	\[
	\frac{\partial\Lambda_d}{\partial B}
	\prod_{\substack{e\mid n\\e>1\\e\neq d}}
	\Lambda_e
	\equiv 0
	\pmod{\Lambda_d}.
	\]
	Since $\Lambda_d$ is coprime to every $\Lambda_e$ with
	$e\neq d$, we obtain
	\[
	\Lambda_d
	\mid
	\frac{\partial\Lambda_d}{\partial B}.
	\]
	However,
	\[
	\deg_B
	\left(
	\frac{\partial\Lambda_d}{\partial B}
	\right)
	<
	\deg_B(\Lambda_d)
	\]
	unless the derivative is zero. Therefore
	\[
	\frac{\partial\Lambda_d}{\partial B}=0.
	\]
	
	Over a field of characteristic $2$, a polynomial in
	$\mathbb F_q[B,C]$ has zero derivative with respect to $B$ if and
	only if every exponent of $B$ occurring in it is even. Hence
	\[
	\Lambda_d(B,C)\in\mathbb F_q[B^2,C].
	\]
	Taking $d=n$ gives
	\[
	\Lambda_n(B,C)\in\mathbb F_q[B^2,C],
	\]
	as desired.
\end{proof}

\begin{corollary}
	If $n >1$ is odd, then for any primitive element $c$ of $\mathbb{F}_{q}$, every root of the induced polynomial $\Lambda_{n}(B,c) \in \mathbb{F}_{q}[B]$ has multiplicity $2$. 
\end{corollary}

\begin{proof}
	Since $\Lambda_{n}(B,C) \in \mathbb{F}_{q}[B^2,C]$, then for any primitive element $c$ of $\mathbb{F}_{q}$, $\Lambda_{n}(B,c) \in \mathbb{F}_{q}[B^2]$. We write $\Lambda_{n}(B,c)$ as
	$$\Lambda_{n}(B,c) = \sum_{i}a_{i}B^{2i}.$$
	As $q$ is a power of $2$, say, $q = 2^e$, setting
	$$\Lambda_{n}^{\prime}(B,c) = \sum_{i}a^{2^{e-1}}B^i \in \mathbb{F}_{q}[B],$$
	then one has
	$$\Lambda_{n}^{\prime}(B,c)^{2} = \sum_{i}a_{i}^{q}B^{2i} = \sum_{i}a_{i}B^{2i} = \Lambda_{n}(B,c).$$
	The conclusion follows immediately.
\end{proof}

Intrinsically, to obtain an optimal determining polynomial for primitive quadratic polynomials over $\mathbb{F}_{q}$, one needs to take a square root of $\Lambda_{N}(B,C)$. However, in general, this cannot be done in $\mathbb{F}_{q}[B,C]$. Instead, we take a square root of $\Lambda_{N}(B,C)$ in the quotient ring
$$A[B] = \mathbb{F}_{q}[B,C]/(\Phi_{q-1}(C)),$$
where $A = \mathbb{F}_{q}[C]/(\Phi_{q-1}(C))$.

The Frobenius map
\[
\operatorname{Fr}:A\longrightarrow A,
\qquad
a\longmapsto a^2,
\]
is an automorphism. Its inverse is given by
\[
\operatorname{Fr}^{-1}(a)=a^{\frac{q}{2}}.
\]
Indeed, for every $a\in A$,
\[
\left(a^{\frac{q}{2}}\right)^2=a^q=a.
\]

\begin{definition}
	\label{def:inverse-Frobenius-reduction}
	Define the inverse-Frobenius operator 
	\[
	\operatorname{Fr}^{-1}:A[B^2]\longrightarrow A[B]
	\]
	by $\operatorname{Fr}^{-1}(B^2) = B$, and $\operatorname{Fr}^{-1}(a) = a^\frac{q}{2}$ for any $a \in A$. Precisely, if
	\[
	F(B,C)=\sum_{i,j}a_{ij}B^{2i}C^j
	\in\mathbb F_q[B^2,C],
	\]
	then, modulo $\Phi_{q-1}(C)$,
	\[
	\operatorname{Fr}^{-1}(F)(B,C)
	=
	\sum_{i,j}
	a_{ij}^{\frac{q}{2}}B^iC^{\frac{jq}{2}}.
	\]
\end{definition}

The operator $\operatorname{Fr}^{-1}$ provides a canonical square root over the coefficient ring $A$.

\begin{lemma}
	\label{lem:S-square}
	For every $F\in A[B^2]$, one has
	\[
	\operatorname{Fr}^{-1}(F)^2=F
	\]
	in $A[B]$.
	Consequently, if $c\in\mathbb F_q^*$ is primitive, then
	\[
	\operatorname{Fr}^{-1}(F)(B,c)^2=F(B,c).
	\]
\end{lemma}

\begin{proof}
	Write
	\[
	F=\sum_i a_iB^{2i},
	\quad a_i\in A.
	\]
	Then
	\[
	\operatorname{Fr}^{-1}(F)
	=
	\sum_i a_i^{\frac{q}{2}}B^i.
	\]
	Since the characteristic is $2$,
	\[
	\begin{aligned}
		\operatorname{Fr}^{-1}(F)^2
		&=
		\sum_i
		\left(a_i^{\frac{q}{2}}\right)^2B^{2i}  \\
		&=
		\sum_i a_i^qB^{2i}
		=
		\sum_i a_iB^{2i}
		=
		F.
	\end{aligned}
	\]
	The specialization statement follows immediately.
\end{proof}

Recall that $N=q+1$ is odd, then $\Lambda_{N}(B,C) \in \mathbb{F}_{q}[B^2,C]$. Denote by $\overline{\Lambda}_{N}(B,C) \in A[B]$ the image of $\Lambda_{N}(B,C)$ under the canonical projection
$$\mathbb{F}_{q}[B,C] \rightarrow \mathbb{F}_{q}[B,C]/(\Phi_{q-1}(C)) = A[B].$$
Clearly $\overline{\Lambda}_{N}(B,C)$ lies in $A[B^2]$. Define $P_{q}(B,C)$ be a preimage of $\operatorname{Fr}^{-1}(\overline{\Lambda}_{N}(B,C))$ in $\mathbb{F}_{q}[B,C]$. If $\Lambda_{N}(B,C) = \sum\limits_{i,j}a_{ij}B^{2i}C^{j}$, then $P_{q}(B,C)$ can be chosen to be 
$$P_{q}(B,C) = \sum_{i,j}a_{ij}^{\frac{q}{2}}B^{i}C^{\frac{jq}{2}}.$$

\begin{theorem}
	\label{thm:char2-optimal}
	Let $q$ be a power of $2$ and let $N=q+1$. For every primitive element $c$ of $\mathbb F_q$, the polynomial $P_q(B,c)\in\mathbb F_q[B]$ is monic and square-free, and its roots are exactly the elements
	$b\in\mathbb F_q$ for which 
	$$X^2+bX+c$$ 
	is primitive over $\mathbb F_q$. In particular, $P_q(B,C)$ is an optimal determining polynomial for
	primitive quadratic polynomials over $\mathbb F_q$.
\end{theorem}

\begin{proof}
	By Lemma~\ref{lem:S-square},
	\[
	P_q(B,C)^2=\Lambda_N(B,C) \in A[B].
	\]
	Thus, for every primitive element $c\in\mathbb F_q$,
	\[
	P_q(B,c)^2=\Lambda_N(B,c).
	\]
	
	Let $b\in\mathbb F_q$ be a root of $P_q(B,c)$. Then
	\[
	\Lambda_N(b,c)=0.
	\]
	Let $\theta_1$ and $\theta_2$ be the roots of
	\[
	f(X)=X^2+bX+c.
	\]
	Since $N=q+1$ is coprime to $q$, Lemma~4.1 gives
	\[
	\operatorname{ord}
	\left(\frac{\theta_1}{\theta_2}\right)
	=
	N=q+1.
	\]
	In particular, $f(X)$ is irreducible. Indeed, if both roots belonged
	to $\mathbb F_q$, then $\frac{\theta_1}{\theta_2}\in\mathbb F_q^*$ would have order dividing $q-1$, which contradicts with
	\[
	\gcd(q-1,q+1)=1.
	\]
	Thus $f(X)$ is irreducible, and its binomial order is $q+1$.
	Since $c$ is primitive in $\mathbb F_q$, Proposition~3.5 implies
	that $f(X)$ is primitive.
	
	Conversely, suppose that
	\[
	f(X)=X^2+bX+c
	\]
	is primitive. Then its binomial order is $q+1=N$. If its two roots
	are $\theta$ and $\theta^q$, then
	\[
	\operatorname{ord}
	\left(\frac{\theta}{\theta^q}\right)
	=N.
	\]
	By Lemma~4.1,
	\[
	\Lambda_N(b,c)=0.
	\]
	Since
	\[
	P_q(B,c)^2=\Lambda_N(B,c),
	\]
	it follows that
	\[
	P_q(b,c)=0.
	\]
	Hence the roots of $P_q(B,c)$ are exactly the coefficients $b$
	for which $X^2+bX+c$ is primitive.
	
	By Lemma~2.10,
	\[
	\deg_B\Lambda_N=\varphi(N).
	\]
	Since
	\[
	P_q(B,c)^2=\Lambda_N(B,c),
	\]
	we have
	\[
	\deg_BP_q(B,c)=\frac{\varphi(N)}2.
	\]
	On the other hand, by Lemma~3.2 there are exactly $\frac{\varphi(N)}2$ elements $b\in\mathbb F_q$ for which $X^2+bX+c$ is primitive.
	Therefore all roots of $P_q(B,c)$ are distinct, and
	$P_q(B,c)$ is square-free.
	
	Finally, since $\Lambda_N(B,C)$ is monic in $B$, its inverse-Frobenius
	reduction is also monic in $B$. Hence $P_q(B,c)$ is monic, and the
	assertion follows.
\end{proof}

Parallel to Proposition \ref{prop:Lucas-polynomial-criterion}, we also present a criterion for primitive quadratic polynomials which
does not require the direct computation of $\Lambda_N(B,C)$.

For $1\leq i\leq s$, the Lucas atomic decomposition gives
\[
\mathcal{U}_{\frac{N}{p_{i}}}(B,C)
=
\prod_{\substack{d\mid \frac{N}{p_{i}}\\d>1}}
\Lambda_d(B,C).
\]
Since every divisor of $N$ is odd, all these polynomials lie in
$\mathbb F_q[B^2,C]$. Define
\[
E_q(B,C)
=
\operatorname{Fr}^{-1}(
\operatorname{lcm}
(\mathcal{U}_{\frac{N}{p_{1}}},\ldots,
\mathcal{U}_{\frac{N}{p_{s}}})).
\]

\begin{proposition}
	\label{prop:char2-practical}
	Let $f(X)=X^2+bX+c\in\mathbb F_q[X]$ be irreducible, where $c$ is
	a primitive element of $\mathbb F_q$. Then $f(X)$ is primitive if
	and only if
	\[
	E_q(b,c)\neq0.
	\]
\end{proposition}

\begin{proof}
	By the Lucas atomic decomposition,
	\[
	\mathcal{U}_{\frac{N}{p_{i}}}(B,C)
	=
	\prod_{\substack{d\mid \frac{N}{p_{i}}\\d>1}}
	\Lambda_d(B,C).
	\]
	Hence
	\[
	\operatorname{lcm}
	(\mathcal{U}_{\frac{N}{p_{1}}},\ldots,\mathcal{U}_{\frac{N}{p_{s}}})
	=
	\prod_{\substack{d\mid N\\1<d<N}}
	\Lambda_d(B,C).
	\]
	Indeed, every proper divisor $d$ of $N$ divides $\frac{N}{p_{i}}$ for at
	least one prime divisor $p_i$ of $N$. Therefore, by the atomic
	decomposition of $\mathcal{U}_N$,
	\[
	\operatorname{lcm}
	(\mathcal{U}_{\frac{N}{p_{1}}},\ldots,\mathcal{U}_{\frac{N}{p_{s}}})
	=
	\frac{\mathcal{U}_N(B,C)}{\Lambda_N(B,C)}.
	\]
	
	Now let $f(X)=X^2+bX+c$ be irreducible. By Proposition~3.5, its
	binomial order is a divisor of $N$, and $f$ is primitive if and only
	if $\operatorname{ord}_b(f)=N$. If $\operatorname{ord}_b(f)<N$, then
	\[
	\operatorname{ord}_b(f)\mid\frac{N}{p_i}
	\]
	for some prime divisor $p_i$ of $N$. Equivalently,
	\[
	\mathcal{U}_{\frac{N}{p_{i}}}(b,c)=0
	\]
	for some $i$, and hence
	\[
	E_q(b,c)=0.
	\]
	Conversely, if $E_q(b,c)=0$, then
	\[
	\mathcal{U}_{\frac{N}{p_{i}}}(b,c)=0
	\]
	for some $i$, so the binomial order of $f$ is a proper divisor of
	$N$. Thus $f$ is not primitive.
\end{proof}

\subsection{Uniqueness of the optimal determining polynomial}
We conclude this section by discussing the uniqueness of the optimal
determining polynomial. There are two natural levels of uniqueness.
First, after the constant term $C$ is specialized to a fixed primitive
element $c\in\mathbb F_q$, the optimal determining polynomial is
uniquely determined. On the other hand, as a polynomial in the two
indeterminates $B$ and $C$, an optimal determining polynomial need not
be unique. Nevertheless, its class in the quotient ring 
$$A[B] = \mathbb{F}_{q}[B,C]/(\Phi_{q-1}(C)),$$
where $A = \mathbb{F}_{q}[C]/(\Phi_{q-1}(C))$, is unique.

For any primitive element $c$ of $\mathbb{F}_{q}$, denote
$$\mathcal{P}_{c} = \{b \in \mathbb{F}_{q} \ | \ X^{2}+bX+c \ \mathrm{is} \ \mathrm{primitive}\}.$$
The order of $\mathcal{P}_{c}$ is exactly $\mathrm{N}_{\mathrm{prim}}(q,c)$, which by Lemma \ref{lem 7} is
\[
\mathrm{N}_{\mathrm{prim}}(q,c)
=
\begin{cases}
	\varphi(q+1), & q\text{ odd},\\[4pt]
	\dfrac{\varphi(q+1)}{2}, & q\text{ even}.
\end{cases}
\]

The first uniqueness statement is immediate from the definition of an
optimal determining polynomial.

\begin{proposition}\label{prop 3}
	\label{prop:fiberwise-uniqueness}
	Let $P_{q}(B,C) \in \mathbb{F}_{q}[B,C]$ be an optimal determining polynomial for primitive quadratic polynomials over $\mathbb{F}_{q}$. Then for any primitive element $c$ of $\mathbb{F}_{q}$,
	$$P_{q}(B,c) = \prod_{b \in \mathcal{P}_{c}}(B-b).$$
\end{proposition}

Proposition \ref{prop 3} says that an optimal determining polynomial is unique on every fiber
corresponding to a primitive value of $C$. We now formulate the
corresponding global uniqueness statement.

Since $\Phi_{q-1}(C)$ splits completely over
$\mathbb F_q$, its roots are precisely the primitive elements of
$\mathbb F_q$. We consider the affine $\mathbb F_q$-scheme
\[
\mathcal X
=
\mathbb A^1_{\mathbb F_q}
\times
\operatorname{Spec}A
=
\operatorname{Spec}A[B].
\]
Its $\mathbb F_q$-rational points are exactly
\[
\mathcal X(\mathbb F_q)
=
\left\{
(b,c):
b\in\mathbb F_q,\;
c\text{ is primitive in }\mathbb F_q
\right\}.
\]
And $A[B]$ is the affine coordinate ring of the scheme $\mathcal{X}$, whose elements can be viewed as regular functions on $\mathcal{X}$. This explains why the uniqueness of the optimal determining polynomials only being valid in $A[B]$ is natural.

\begin{theorem}
	\label{thm:global-uniqueness}
	Let $P_q(B,C),P'_q(B,C)\in\mathbb F_q[B,C]$ be two optimal
	determining polynomials for primitive quadratic polynomials over
	$\mathbb F_q$. Then
	\[
	P_q(B,C)\equiv P'_q(B,C)
	\pmod{\Phi_{q-1}(C)}.
	\]
	Equivalently, they have the same images in the quotient ring $A[B]$. Consequently, the optimal determining polynomial determines a unique
	regular function on $\mathcal X$.
\end{theorem}

\begin{proof}
	By Proposition~\ref{prop:fiberwise-uniqueness}, for every primitive
	element $c\in\mathbb F_q$ one has
	\[
	P_q(B,c)=P'_q(B,c)
	\]
	in $\mathbb F_q[B]$.
	
	Write
	\[
	P_q(B,C)-P'_q(B,C)
	=
	\sum_{j=0}^d a_j(C)B^j,
	\qquad
	a_j(C)\in\mathbb F_q[C].
	\]
	For every primitive element $c\in\mathbb F_q$, the identity
	\[
	P_q(B,c)-P'_q(B,c)=0
	\]
	in $\mathbb F_q[B]$ implies
	\[
	a_j(c)=0
	\]
	for every $j$.
	
	The primitive elements of $\mathbb F_q$ are precisely the roots of
	$\Phi_{q-1}(C)$. Since $\Phi_{q-1}(C)$ is separable, it follows that
	\[
	\Phi_{q-1}(C)\mid a_j(C)
	\]
	for every $j$. Therefore
	\[
	P_q(B,C)-P'_q(B,C)
	\in
	(\Phi_{q-1}(C)),
	\]
	and hence
	\[
	P_q(B,C)\equiv P'_q(B,C)
	\pmod{\Phi_{q-1}(C)}.
	\]
	This proves the assertion.
\end{proof}

\section{A first-zero coefficient approach to binomial order and order}\label{sec 5}
Let $f(X) = X^{2}+bX+c \in \mathbb{F}_{q}[X]$ be an irreducible polynomial, and $\theta$ be a root of $f(X)$ lying in $\mathbb{F}_{q^2}$. Since $f(X)$ is irreducible, another root of $f(X)$ is $\theta^{q}$, and Vieta's formula gives
$$\theta^{q+1} = \theta\theta^q = c.$$
As $f(X)$ is the minimal polynomial of $\theta$ over $\mathbb{F}_{q}$, one has
$$f(X) \mid X^{q+1}-c.$$

We denote by $H_{f}(X)$ the quotient
$$H_f(X) =\frac{X^{q+2}-cX}{f(X)} =h_1X^q+h_2X^{q-1}+\cdots+h_qX+h_{q+1}.$$
As $f(X) \mid X^{q+1}-c$, clearly $h_{q+1} = 0$.

The main theorem of this section provides an approach to the binomial order and the order of $f(X)$ via the coefficients $h_{1},h_{2},\cdots,h_{q+1}$ of $H_{f}(X)$, which generalizes Theorem $4$ in \cite{VegaNecessary}. 

\begin{theorem}\label{thm 6}
	With the notations defined as above, let
	$$r = \mathrm{min}\{2\leq i \leq q+1: \ h_{i}=0\}$$
	and $\lambda = -ch_{r-1}$. Then the binomial order of $f(X)$ is $r$ and the minimal binomial multiple of $f(X)$ is $X^{r}-\lambda$. Consequently, the order of $f(X)$ is $r\cdot\mathrm{ord}(\lambda)$.
\end{theorem}

As a preparation, we first consider the following lemma.

\begin{lemma}\label{lem 11}
	Let $\{U_{i}\}_{i \geq 0}$ be the linear recurrence sequence associated to $f(X)$. Then $h_{i}=U_{i}$ for any $1 \leq i \leq q+1$.
\end{lemma}

\begin{proof}
	By the definition of $H_{f}(X)$ we have
	\begin{equation}\label{eq 6}
		X^{q+2}-cX = (X^{2}+bX+c)(h_1X^q+h_2X^{q-1}+\cdots+h_qX+h_{q+1}).
	\end{equation}
	Comparing the coefficients on both sides of \eqref{eq 6} yields that $h_{1}=1$, $h_{2}=-b$, and
	$$h_{i} = -bh_{i-1} - ch_{i-2}, \ \forall 3 \leq i \leq q+1.$$
	Noting that $U_{2} = -bU_{1}-cU_{0} = -b$, we obtain by induction that
	$$h_{i} = U_{i}, \ \forall 1 \leq i \leq q+1.$$
\end{proof}

Now we give the proof of Theorem \ref{thm 6}

\begin{proof}[Proof of Theorem \ref{thm 6}]
	First since $h_{q+1}=0$, there exists a smallest integer $2 \leq r \leq q+1$ such that $h_{r}=0$, and by this requirement one has $\lambda = -ch_{r-1} \neq 0$. Combining Corollary \ref{coro 1} and Lemma \ref{lem 11}, we have $\mathrm{ord}_{\mathrm{b}}(f) = r$ and $X^{r}-\lambda$ is exactly the minimal binomial multiple of $f(X)$. It follows from Proposition \ref{prop 1} that the order of $f(X)$ is $r\cdot\mathrm{ord}(\lambda)$.
\end{proof}

\begin{corollary}
	If none of $h_{1},\cdots,h_{q}$ is zero, then
	$$\mathrm{ord}_{\mathrm{b}}(f) = q+1, \quad \mathrm{ord}(f) = (q+1)\cdot\mathrm{ord}(c),$$
	and the minimal binomial multiple of $f(X)$ is $X^{q+1}-c$. If, furthermore, $c$ is a primitive element of $\mathbb{F}_{q}$, then $f(X)$ is primitive.
\end{corollary}

\begin{proof}
	By Theorem \ref{thm 6} the binomial order of $f(X)$ is $q+1$, and from the identity \eqref{eq 6} it is trivial to check $h_{q}=-1$. Thus by Theorem \ref{thm 6} again the minimal binomial multiple of $f(X)$ is $X^{q+1}-c$, and the order of $f(X)$ is $(q+1)\cdot\mathrm{ord}(c)$. The last assertion follows immediately form Lemma \ref{lem 4}.
\end{proof}

\section*{Acknowledgment}
The first author was supported by Basic Research Program Young Scientists Guidance Project of Guizhou Province (QN[2025]186).

\section*{Data availability}
Data sharing not applicable to this article as no datasets were generated or analysed during the current study.

\section*{Declaration of competing interest}
The authors declare that we have no known competing financial interests or personal relationships 
that could be perceived to influence the work reported in this paper.

\end{document}